\documentclass[12pt]{article}

\usepackage{amssymb,amsmath,epsfig,amsthm, mathdots}
\usepackage{tikz}
\usetikzlibrary{positioning,decorations.pathreplacing,matrix,arrows,calc,shapes,arrows}

\usepackage{fullpage}

\usepackage{epsfig}
\usepackage{amssymb}
\usepackage{hyperref}
\newtheorem{theorem}{Theorem}[section]
\newtheorem{prop}[theorem]{Proposition}
\newtheorem{lemma}[theorem]{Lemma}

\newtheorem{conjecture}[theorem]{Conjecture}

{\theoremstyle{remark}
\newtheorem{remark}[theorem]{Remark}

}

\newcommand{\CC}{\mathcal{C}}
\newcommand{\Av}{\operatorname{Av}}
\newcommand{\bonds}{\operatorname{bond}}
\newcommand{\lead}{\ell}
\newcommand{\lrmin}{\operatorname{lrmin}}
\newcommand{\Si}{\operatorname{Si}}
\newcommand{\Y}{\mathcal{Y}}

\newcommand{\T}{\mathcal{T}}

\newcommand{\id}{\operatorname{id}}

\newcommand{\makeset}[2]{ \{#1\;|\;#2\} }
\newcommand{\LR}[1]{ LR\left(#1\right) }

\usepackage{tikz}

\definecolor{light-gray}{gray}{0.6}
\definecolor{dark-gray}{gray}{0.4}

\numberwithin{equation}{section}

\usepackage{chngcntr}
\counterwithin{figure}{section}

\begin{document}

\title{Egge triples and unbalanced Wilf-equivalence}

\date{}

\author{
Jonathan Bloom\\
Lafayette College\\
Department of Mathematics\\
Easton, PA 18042, USA\\
{\tt bloomjs@lafayette.edu}
\and 
Alexander Burstein\\
Howard University\\
Department of Mathematics\\
Washington, DC 20059, USA\\
{\tt aburstein@howard.edu}
}

\maketitle

\begin{abstract}
Egge~\cite{Egge} conjectured that permutations avoiding the set of patterns $\{2143,3142,\tau\}$, where $\tau\in\{246135,254613,524361,546132,263514\}$, are enumerated by the large Schr\"oder numbers (and thus $\{2143,3142,\tau\}$ with $\tau$ as above is Wilf-equivalent to the set of patterns $\{2413,3142\}$). Burstein and Pantone~\cite{BP} proved the case of $\tau=246135$. We prove the remaining four cases. As a byproduct of our proof, we also enumerate the case $\tau=4132$.
\end{abstract}

\maketitle

\section{Introduction}

Writing a permutation $\pi\in S_n$ as $\pi_1\ldots \pi_n$ with \emph{length} $ |\pi|=n$,  we say $\pi$ \emph{avoids}  $\tau\in S_k$ if no subsequence of $\pi$ is order isomorphic to $\tau$.  In this setting we refer to $\tau$ as a \emph{pattern}.  For a set of patterns $T$, we denote by $\Av_n(T)$ the set of $\pi\in S_n$ that simultaneously avoid all of $T$. In this case we say $T$ is a \emph{basis} for $\Av(T)$.  We set $\Av(T) = \cup_{n\geq 0} \Av_n(T)$.   

Two sets of patterns, $T_1$ and $T_2$, are said to be \emph{Wilf-equivalent}  provided $|\Av_n(T_1)|=|\Av_n(T_2)|$ for all $n\ge 0$. Moreover, a Wilf-equivalence is said to be \emph{unbalanced} if, for some $k\ge 0$, $T_1$ and $T_2$ contain a different number of patterns of length $k$. Until recently, the only documented unbalanced Wilf-equivalence (see~\cite{AMR}) was 
\[
\{1342\}\qquad\textrm{and}\qquad \makeset{2,2m-1,4,1,6,3,8,5,\dots,2m,2m-3}{m=2,3,4,\dots},
\]
which involves a finite and an infinite set. (And, clearly, allowing both sets $T_1$ and $T_2$ to be infinite, makes it trivial to recursively construct unbalanced Wilf-equivalences.) In 2012, though, Egge~\cite{Egge},   made the following conjecture at the AMS Fall Eastern Meeting.

\begin{conjecture}[Egge, 2012] \label{conj:egge}
Fix $\tau\in\{246135,254613,524361,546132,263514\}$.  Then
\begin{align*}
\sum_{n\geq 0} |\Av_n(2143,3142,\tau)|x^n = \frac{3-x-\sqrt{1-6x+x^2}}{2},	
\end{align*}
i.e., $\Av_n(2143,3142,\tau)$ is counted by the large Schr\"oder numbers for $n\ge 0$.  
\end{conjecture}

As the large Schr\"oder numbers count several other pattern-avoiding classes, for example, $\{2413,3142\}$ or $\{1243,2143\}$ (see Kremer and Shiu~\cite[Table 1]{KS}), Egge's conjecture implies several examples of this unbalanced phenomenon where both sets are finite.  That said, other such examples have recently been proved.  In~\cite{BP} Burstein and Pantone show that $\{1234\}$ and $\{1324,3416725\}$ are yet another unbalanced Wilf-equivalence.  (This instance was independently shown by Jel\'{\i}nek in his doctoral thesis~\cite{Jelinek} that remains unpublished.)  Moreover, in the same paper, Burstein and Pantone prove the case for $\tau = 246135$.  

We would be remiss if we did not mention another motivation behind Egge's conjecture.  Consider the two enumeration sequences, listed by Kremer and Shiu~\cite[Table 1]{KS}:
\bigskip
\begin{center}
$
\begin{array}{r|c c c c c c c}
\hline
n= & 2 & 3 & 4 & 5 & 6 & 7 &\ldots \\
\hline
\Av_n(2143,3142) &  2 & 6 & 22& 90 & 395 & 1823 &\ldots\\
n \textrm{th large Schr\"oder \#} & 2 & 6&22&90&394 & 1806&\ldots
\end{array}$
\end{center}
\bigskip
As these two sequences are ``almost'' the same, they beg the following question:  What (if any) singleton patterns of length 6 can be added to $\{2143,3142\}$ so that the resulting class is enumerated by the large Schr\"oder numbers?  Egge's conjecture suggests 5 such patterns and computational evidence confirms that these 5 values of $\tau$, plus their $180^\circ$ rotations, are the only patterns that yield the large Schr\"oder numbers.

As noted above, Burstein and Pantone previously proved the case when $\tau =246135$.   In this paper, we prove the remaining four cases.   We structure the paper as follows.  In the next two sections we prove the cases $\tau = 254613, 524361$ and $546132$.  In each of these cases, we demonstrate a decomposition of the permutations in $\Av_n(2143,3142,\tau)$ based on left-to-right maxima (LR-maxima).  We then translate this combinatorial decomposition into a functional equation, which we solve using the kernel method.    In our last section we tackle the remaining pattern $263514$ whose enumeration is based on the idea of simple permutations.  It should be noted that Burstein and Pantone in~\cite{BP} enumerated the case of $\tau = 246135$ using simple permutations as well.  

In the four cases we consider in this paper we will demonstrate that the generating function for the large Schr\"oder numbers is a solution to a certain functional equation.  We note here that this is actually sufficient; we do not need to address uniqueness in each case. The reason for this is that the recursive nature of our functional equations implies that any solution is determined by the first few terms of its Taylor series.  Therefore, any solution found whose Taylor series begins with $1+x+2x^2+6x^3$ must be the \emph{only} solution whose Taylor series begins in this manner.    

We occupy the remainder of this section with some basic definitions and lemmas that will be used throughout the sequel.  For brevity, we set $\CC(\tau) = \Av(2143,3142,\tau)$ and use standard subscript notation to refine this set (and any other) by length.

As mentioned above, many of our arguments will lean heavily on the idea of LR-maxima.  In particular, an \emph{LR-maximum} in a permutation $\pi$ is an index $i$ such that  $\pi_j<\pi_i$ for all $j<i$.  We denote by $\LR{\pi}$ the set of all LR-maxima in $\pi$.  Denoting $\LR{\pi} = \{i_1<\cdots <i_s\}$, we observe that 
$i_1 = 1$, $\pi_{i_s} = n$, $\pi_{i_1} < \cdots <\pi_{i_s}$, and $ \pi_{i_j}>\pi_x$ provided $i_j<x<i_{j+1}$.  We define a \emph{horizontal gap} to be an index $i\in \LR{\pi}$ such that $i<n$ and $i+1\notin \LR{\pi}$.  In other words, a horizontal gap is an index $i$ that is both a descent and a LR-maximum. We define a \emph{leading maximum} to be an index $i\in \LR{\pi}$ such that $\{1,\ldots, i\} \in \LR{\pi}$.  We denote the number of leading maxima in $\pi$ by $\lead(\pi)$.  Note that $n$ is a leading maximum if and only if $\pi = 1\ldots n$. In this case, $\ell(\pi) = n$ as well. 

Given any two permutations $\alpha,\beta$ we define $\alpha\oplus \beta$ and $\alpha\ominus\beta$ to be the classical sum and skew-sum, respectively.  It will also be convenient to have a slightly more general version of ``skew-sum''.  To this end, let $i$ be a leading maximum of $\beta$ and define $\alpha\ominus_i \beta$ to be the permutation obtained from $\alpha\ominus \beta$ by sliding the first $i$ leading maxima to the left of $\alpha$.  Pictorially, we have

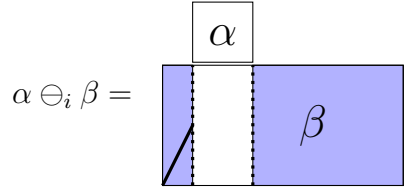
\begin{center}
\begin{tikzpicture}[scale=0.4, baseline = (current bounding box.center)]

\begin{scope}[shift={(0,0)}]
\node at (-3,3) {$\alpha\ominus_i\beta=$};

\draw[fill = blue!30!white] (0,0) rectangle (1,4);
\draw[fill = blue!30!white] (3,0) rectangle (8,4);
\draw (0,0) rectangle (8,4);
\draw (1.0,4.1) rectangle (3,6.1);

\node at (2,5) {\Large $\alpha$};
\draw[very thick] (0,0) -- (1,2);
\draw[very thick,dotted] (1,0)--(1,4);
\draw[very thick,dotted] (3,0)--(3,4);
\node at (5,2) {\Large $\beta$};
\end{scope}
\end{tikzpicture}
\end{center}
where the diagonal line represents the $i$ LR-maxima that have been ``extracted" to the left of $\alpha$.  Consequently, we refer to this construction as {\tt extraction}.  (Although not depicted in the diagram, it should be stressed that only the positions and not the values of the first $i$ LR-maxima  in $\beta$ change.)   

An observation that will be useful in what follows is if $\pi$ is any permutation with exactly one horizontal gap, then
$$\pi = (1\ominus_i \beta)\oplus 1\ldots m,$$
for some $m, i$, and $\beta$ such that $i\le\ell(\beta)$ and $i<|\beta|$.  

As the next lemma shows, the idea of {\tt extraction} also plays nicely with the patterns of interest.  As the proof of this lemma follows  directly from the definitions involved, we omit its proof. 

\begin{lemma}\label{lem:extraction} 
Fix $\tau\in\{254613, 524361,546132\}$ and $\beta\in\CC(\tau)$, then 
$$1\ominus_i \beta\ \in \CC(254613),$$ where $i\le\ell(\beta)$.
\end{lemma}

As all the permutation classes of interest contain both 2143 and 3142 in their basis, we next prove a couple of lemmas that illuminate the structure of $(2143,3142)$-avoiding permutations.

\begin{lemma} \label{lem:2143}
Fix $\pi \in \Av_n(2143)$ and $\ell = \lead(\pi)$. For any $i\notin \LR{\pi}$, we have $\pi_{\ell} > \pi_i$. In other words, 
$$
\{\pi_i\ |\ \ell\leq i \textrm{ and } i\in\LR{\pi}\}=\{\pi_\ell,\pi_{\ell+1}, \ldots, n\}.
$$

\end{lemma}

\begin{proof}
Fix $\pi\in \Av_n(2143)$ and let $\ell$ be the largest leading maximum of $\pi$.  If $\ell = n$, then  $\pi = 1\ldots n$ and there is nothing to prove. On the other hand, if $\ell<n$, then, as $\ell$ is the largest leading maximum,  $\pi_\ell>\pi_{\ell+1}$, i.e., $\ell$ is a descent. Now assume for a contradiction that there is some index $j\notin \LR{\pi}$ with the property that  $\pi_\ell < \pi_j$.   As $\ell$ is a descent, it follows that $\ell+1<j$.  As $j$ is not a LR-maximum, there must exist some $k\in \LR{\pi}$ such that $\ell+1<k<j$.  This immediately implies that $\pi_\ell\pi_{\ell+1}\pi_k\pi_j$ is an occurrence of $2143$, a contradiction. 
\end{proof}

\begin{remark}
It follows from the proof of Lemma~\ref{lem:2143} that if $\pi\neq 1\ldots n$, then $\ell$ is both an LR-maximum and a descent, and hence also a horizontal gap.  We will tacitly use this fact in the sequel.  
\end{remark}

If our permutations also avoid $3142$, we can say more.  

\begin{lemma} \label{lem:2143-3142}
Every $\pi\in\Av_n(2143,3142)$ decomposes as in Figure~\ref{fig:2143_4132_decomp}.  That is, if $\pi$ has $g$ horizontal gaps and if $\beta^i$ is the permutation strictly in the $i$th horizontal gap, then 
$$\beta^1\ominus \beta^2\ominus\cdots \ominus \beta^g,$$
is the permutation obtained by removing all the LR-maxima in $\pi$.  
\end{lemma}

\begin{figure}
\begin{center}
\begin{tikzpicture}[scale=0.25]

\begin{scope}[shift={(0,0)}]
\draw[very thick] (-1.5,0) -- (.5,15);
\node at (1,16) {$\times$};

\draw[fill = blue!30!white] (2,11) rectangle (6,15);
\node at (4,13) {$\beta^1$};

\draw[very thick] (6.5,16) -- (8.5,16.5) node[pos=1.3] {$\times$};

\draw[fill = blue!30!white] (10,7) rectangle (14,11);
\node at (12,9) {$\beta^2$};

\draw[very thick] (14.5,16) -- (16.5,16.5) node[pos=1.3] {$\times$};

\draw[fill = blue!30!white] (22,0) rectangle (26,4);
\node at (24,2) {$\beta^g$};

\node at (19,6) {\Huge $\ddots$};
\node at (21,18) {\Huge $\ldots$};

\draw[very thick] (27,18) -- (29,18.5);

\draw[very thick] (-2,-1) -- (31,-1);
\draw(1,-0.5) -- (1,-1.5) node[below] {$\ell$};
\end{scope}
\end{tikzpicture}
\caption{The decomposition of $\pi\in\Av(2143,3142)$. Each $\beta^{i}$, along with the leading maxima that lie between the values of $\beta^{i}$, is an arbitrary element of $\Av(2143,3142)$.}
\label{fig:2143_4132_decomp}
\end{center}
\end{figure}
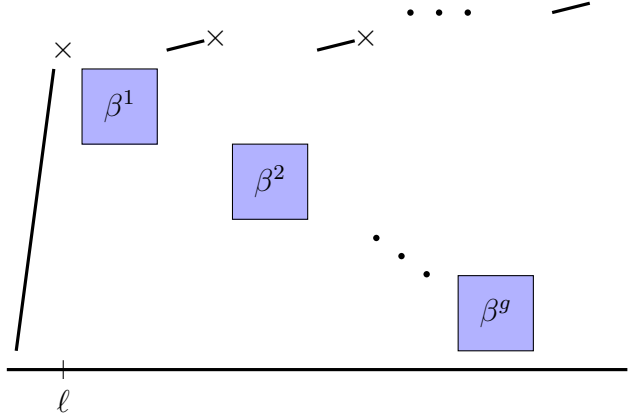

\begin{proof}
Let $i<j$ be horizontal gaps.  If $x$ is a position in the $i$th gap and $y$ is a position in the $j$th gap, then it will suffice to show that $\pi_x>\pi_y$.  By Lemma~\ref{lem:2143} we know that both $\pi_i,\pi_j$ are larger than $\pi_x$ and $\pi_y$, as $\pi_\ell\leq \pi_i<\pi_j$, where $\ell = \lead(\pi)$.  As $\pi_{i}\pi_x\pi_{j}\pi_y$ cannot be an occurrence of 3142, this now forces $\pi_x>\pi_y$.
\end{proof}

\section{The class $\CC(254613)$} \label{sec:254613}

To show that the permutations $\pi\in \CC(254613) = \Av(2143,3142, 254613)$ are counted by the Schr\"oder numbers, we consider three cases depending on the number of horizontal gaps in $\pi$.  In each case, we give a  decomposition of $\pi$, and together, these translate into a functional equation we then solve.  To begin, let us fix $\pi\in \CC(254613)$ and $\ell = \lead(\pi)$.

\bigskip

\texttt{Case 1:} $\pi$ has no horizontal gaps.  

\medskip

In this case $\pi = 1\dots n$ and $\ell = n$.  
\qed

\bigskip

\texttt{Case 2:} $\pi$ has exactly 1 horizontal gap.

\medskip

Any permutation with exactly 1 horizontal gap must be of the form

\begin{equation}\label{eq:254613-case2}
\left(1\ominus_{\ell-1} \beta\right) \oplus 1\ldots m,
\end{equation}
for some permutation $\beta$.  Moreover, it follows from Lemma~\ref{lem:extraction}  that any permutation constructed in this way where $\beta\in\CC_{\geq \ell}(254613)$ and $\lead(\beta)<\ell$ is an element of $\CC(254613)$ with exactly one horizontal gap.  
\qed

\bigskip

\texttt{Case 3:} $\pi$ has at least 2 horizontal gaps.

\medskip

In this case, let us concentrate on the rightmost horizontal gap in Figure~\ref{fig:2143_4132_decomp} and in particular the block $\beta^g$ and its corresponding LR-maxima.  We claim that $\pi$ further decomposes as in Figure~\ref{fig:case3} where $\alpha = \beta^1\ominus\ldots\ominus\beta^{g-1}$ and $\beta^g$ itself decomposes as a skew-sum of   permutations in $\CC(254613)$. We refer to these permutations as \emph{blocks}.    To prove this claim, we now make use of the fact that $\pi$ avoids 254613.  In particular,  consider indices $x<\ell<y<z<w$, where $y$ is the rightmost horizontal gap, $z,w\notin\LR{\pi}$, and the value of $\pi_x$ is between the values $\pi_z$ and $\pi_w$. (These indices have been marked in Figure~\ref{fig:case3} for reference.)   It will now suffice to show that $\pi_z>\pi_w$.  

As $\ell$ is a horizontal gap, $\ell$ is a descent, so $\pi_\ell> \pi_{\ell+1}$.  As $\ell<y$ are LR-maxima, we have $\pi_\ell<\pi_y$.  Moreover, Lemma~\ref{lem:2143-3142} guarantees that $\pi_{\ell+1}$ is greater than both $\pi_z$ and $\pi_w$.  Thus, the only way for 
$$\pi_x\ \pi_\ell\ \pi_{\ell+1}\ \pi_y\ \pi_z\ \pi_w$$
to not be an occurrence of 254613 is for $\pi_z>\pi_w$.   
\qed

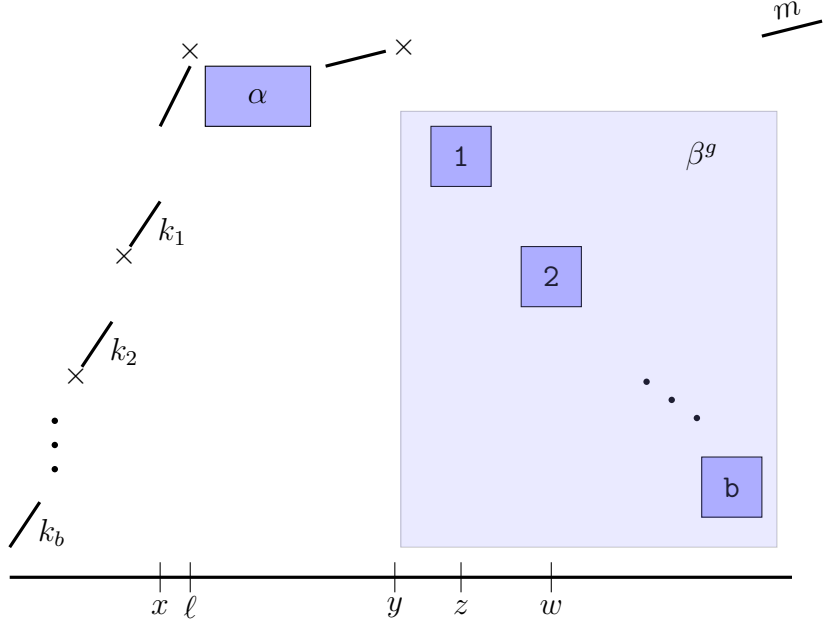
\begin{figure}[!ht]
\begin{center}
\begin{tikzpicture}[scale=0.4]

\begin{scope}[shift={(0,0)}]

\draw[fill = blue!30!white] (.5,0) rectangle (4,2);
\node at (2.25,1) {$\alpha$};
\node at (0,2.5) {$\times$};
\draw[very thick] (-1, 0) -- (0,2);
\draw[very thick] (4.5,2) -- (6.5,2.5) node[pos=1.3] {$\times$};

\begin{scope}[shift={(19,3)}]
\draw[very thick](0,0) -- (2,.5) node[midway,sloped,above] {$m$};
\end{scope}

\draw[fill = blue!30!white] (8,-2) rectangle (10,0);
\node at (9,-1) {\tt 1};

\begin{scope}[shift={(-2,-4)}]
\draw[very thick](0,0)-- node[pos=-.2]{$\times$} ++ (1,1.5);
\node at (1.4,.5) {$k_1$};
\end{scope}

\draw[fill = blue!30!white] (11,-6) rectangle (13,-4);
\node at (12,-5) {\tt 2};

\begin{scope}[shift={(-3.6,-8)}]
\draw[very thick](0,0)-- node[pos=-.2]{$\times$} ++ (1,1.5);
\node at (1.4,.5) {$k_2$};
\end{scope}

\node at (16,-8.5) {\Huge $\ddots$};
\node at (-4.5,-10) {\Huge $\vdots$};

\draw[fill = blue!30!white] (17,-13) rectangle (19,-11);
\node at (18,-12) {\tt b};

\begin{scope}[shift={(-6,-14)}]
\draw[very thick](0,0)--(1,1.5);
\node at (1.4,.5) {$k_b$};

\end{scope}

\draw[very thick] (-6,-15) -- (20,-15);
\draw (-1,-15.5) -- (-1,-14.5);
\draw (0,-15.5) -- (0,-14.5);
\draw (6.8,-15.5) -- (6.8,-14.5);
\draw (9,-15.5) -- (9,-14.5);
\draw (12,-15.5) -- (12,-14.5);
\node at (-1,-16) {$x$};
\node at (0,-16) {$\ell$};
\node at (6.8,-16) {$y$};
\node at (9,-16) {$z$};
\node at (12,-16) {$w$};

\filldraw[fill=blue!40!white, draw=black,opacity=.2] (7,-14) rectangle (19.5,.5);
\node at (17, -1) {$\beta^g$};

\end{scope}
\end{tikzpicture}
\caption{The decomposition of $\pi\in\CC(254613)$ with at least two horizontal gaps. This is a refinement of Figure~\ref{fig:2143_4132_decomp} where $\alpha = \beta^{1}\ominus\ldots\ominus \beta^{g-1}$ and the block labeled $\beta^{g}$ in Figure~\ref{fig:2143_4132_decomp} is shown in greater detail.}

\label{fig:case3}
\end{center}
\end{figure}

To inductively build such a $\pi$ we have two choices.  If we start with a permutation $\pi'$ with at least two gaps, then we may insert a new block into the existing rightmost gap. To do this, we first construct $1\oplus \pi'$ and then insert any element of $\CC(254613)$ as the $\mathtt{(b+1)}$-st block so that it lies southeast of the existing blocks and below the 1 in $1\oplus \pi'$.  (The insertion of the 1 is required to ``separate'' the new block from the exisiting blocks.)  We then take the resulting permutation and direct-sum it with an increasing sequence of length $k_{b+1}$.  We will refer to this construction as {\tt block insertion}.

 On the other hand, if the permutation $\pi'$ we start with has at most one horizontal gap, then we may ``append" a new horizontal gap as follows.  First construct 
$$\pi''=(\pi'\oplus 1)\ominus \beta^1$$
and then add an initial increasing sequence of length $k_1$ and a trailing increasing sequence of length $m$ to obtain the desired permutation
$$(1\ldots k_1)\oplus \pi''\oplus 1\ldots m.$$  We refer to this construction as {\tt gap insertion}.


At first consideration it might seem peculiar that we require permutations to have at least two horizontal gaps for {\tt block insertion}.  Without this restriction, though, not all permutations would have a unique construction.  For example, the permutation 243156 would arise via {\tt Case 2} as  
$$243156 = (1\ominus_1 231) \oplus 12,$$ 
and via block insertion in {\tt Case 3} applied to $13245$. 

We now translate these decompositions into a functional equation.  First, set
$$A_1(t,x) = \sum_{\pi \in \CC(254613)} x^{|\pi|}t^{\lead(\pi)}.$$
It now follows from the above decomposition that this generating function satisfies the following functional equation:
\begin{align}\label{eq:254613-fct-eq}
A_1(t,x) &= \frac{1}{1-tx} + \frac{txE}{1-x} + \Big(A_1-\frac{1}{1-tx}\Big)\Bigg(\frac{x(B-1)}{(1-x)(1-tx)}\Bigg)\Bigg(\frac{1}{1-\frac{tx(B-1)}{1-tx}}\Bigg),
\end{align}
where 
$$E(t,x) = \frac{B-tA_1}{1-t}-\frac{1}{1-tx} \qquad \textrm{and}\qquad B = A_1(x,1).$$
To be clear, the first two terms correspond to the {\tt Case 1} and {\tt Case 2} respectively, where $E(t,x)$ encodes the  {\tt extraction} operation in {\tt Case 2}. (Subtracting $\frac{1}{1-xt}$ in the computation of $E$ guarantees that the result of {\tt extraction} has exactly one horizontal gap.)   The third term encodes {\tt gap and block insertion}. Specifically, the first factor represents permutations with at least one horizontal gap.  The second factor then encoded {\tt gap insertion}, while the third factor populates this new gap with an arbitrary number of additional blocks.  

To solve (\ref{eq:254613-fct-eq}), we use the kernel method.  Collecting terms, we obtain
\begin{align}\label{eq:254613-fct-eq-2}
\left( \frac{B t^3 x^2+B t^2 x^2-B t^2 x-B t x^2+B x-t^2 x+t-1}{(1-t) (1-x) (1-Bt x)}\right)A_\ast=\frac{xt}{1-x}\left(
\frac{B t x-B+1}{(t-1) (t x-1)}
\right)
\end{align}
where $\displaystyle A_\ast = A_1-\frac{1}{1-xt}$.  Setting our kernel to zero, we obtain
\begin{equation}\label{eq:kernel-1}
0=B t^3 x^2+B t^2 x^2-B t^2 x-B t x^2+B x-t^2 x+t-1.
\end{equation}
Explicitly solving for $t$ using a CAS leads to an intractable expression.  To avoid this problem, let $t(x)$ be the desired solution.  Now, setting $t=t(x)$ in the RHS of (\ref{eq:254613-fct-eq-2}), we see that
$0=B x t(x) -B+1$, or equivalently, that
\begin{equation}\label{eq:kernel-2}
xt(x) = \frac{B-1}{B}
\end{equation}

If we then use (\ref{eq:kernel-2}) to reduce our kernel (\ref{eq:kernel-1}), we obtain the cubic 
$$B^3 x + B^2 x^2- 3 B^2 x-B^2+B x+3 B-2=(xB-1)(B^2+(x-3)B+2),$$ whose  solution $B$ such that $g(0)=1$ is 
$$B= \frac{3-x-\sqrt{1-6 x+x^2}}{2}.$$
As $B=A_1(x,1)$, this clearly demonstrates that the $\CC(254613)$ is counted by the large Schr\"oder numbers. Note also that the kernel method solution $t=t(x)$ is then the generating function for the little Schr\"oder numbers.

\section{The classes $\CC(524361)$ and $\CC(546132)$} \label{sec:524361_546132}

As we will see, the decomposition used to enumerate $\CC(546132)$ is a ``mirror'' image of the decomposition given in the $524361$ case.  Consequently, we tackle both enumerations in this section.  As in the previous section, these decompositions strongly depend on LR-maxima.   Additionally, as the pattern $524361$ contains the subsequence $5243$ and the pattern $546132$ ends with the subsequence $6132$, the pattern $4132$ plays an important role in the decompositions in this section.  From this perspective, the characterization of permutations in $\CC(4132)=\Av(2143,3142,4132)$, given in the next lemma, provides a natural starting place.

\begin{lemma}\label{lem:char:4132}
$\CC(4132)$ is precisely the set of permutations $\pi$ such that deletion of their leading maxima results in an element of $\Av(132)$.
\end{lemma}
\begin{proof}
For concreteness, let $\Y$ be the set of permutations $\pi$ with the property that deleting $\pi_i$ for all leading maxima $i$ gives a 132-avoiding permutation.

Fix $\pi\notin \Av(2143,3142,4132)$. Then one of the three prohibited patterns must occur in $\pi$ and, moreover, any such occurrence can include at most one LR-maximum of $\pi$.  Further, all three patterns have the property that deleting their first value results in the pattern 132.    Therefore  removing all the leading maxima in $\pi$ results in a permutation that contains an occurrence of 132. So $\pi\notin \Y$.   

To see the reverse inclusion, consider $\pi\in\Av(2143,3142,4132)$.  If $\pi\in \Av(132)\subset \Y$, we are done.  Otherwise, let $x,y,z$ 
be indices such that $\pi_x\pi_y\pi_z$ is a 132-pattern.   It will now suffice to show that $x$ is a leading maximum. To this end, let $j<i<x$.  A straightforward check reveals that $\pi_j<\pi_i<\pi_x$, as otherwise we would create either a 4132-, 3142-, or 2143-pattern.  Therefore, $\pi_1<\cdots<\pi_x$ and $x$ is a leading maximum.  
\end{proof}

Before continuing, we let
$$C(x) = \frac{1-\sqrt{1-4x}}{2x},$$
i.e., $C(x)$ is the generating function for the Catalan numbers.

\subsection{The class $\CC(524361)$}

In order to show the equality 
$$\sum_{n\geq 0} |\CC_n(524361)|\ x^n = \frac{3-x-\sqrt{1-6x+x^2}}{2},$$
we seek a natural decomposition of the permutations in $\CC_n(524361)=\Av(2143,3142,524361)$.  To begin, fix such a permutation $\pi$ and recall from Lemma~\ref{lem:2143-3142} that $\pi$ decomposes as in Figure~\ref{fig:2143_4132_decomp}. Our decomposition of $\pi$, which we describe next, is guided by the following observation.  In Figure~\ref{fig:2143_4132_decomp},  the blocks labeled $\beta^1,\ldots, \beta^{g-1}$ must be 132-avoiding permutations.  Otherwise, any block that contains a 132, along with $\pi_\ell$, the last horizontal gap, and $\beta^g$ would create an occurrence of 524361.)  Therefore the only block that may contain an occurrence of 132 is $\beta^g$.  Consequently we consider the following two cases.

\medskip

{\tt Case 1:} $\beta^g\in \Av(132)$

\medskip

In this case $\beta^1,\ldots, \beta^g\in \Av(132)$ and hence everything strictly to the left of $\ell$ in Figure~\ref{fig:2143_4132_decomp} must also avoids $132$.   Lemma~\ref{lem:char:4132} now guarantees that  $\pi\in\CC(4132)\subset \CC(524361)$.    
\qed

\bigskip

{\tt Case 2:} $\beta^g\notin \Av(132)$

\medskip

First observe that this case could only occur if we have at least one horizontal gap.  If $\pi$ has exactly one horizontal gap, then 
$$\pi = (1\ominus_i \beta) \oplus 1\ldots m,$$
where $i\leq \lead(\beta)$ and $\beta_{i+1}\beta_{i+2}\ldots\in\CC(524361)\setminus \CC(132)$.

If $\pi$ has at at least two horizontal gaps, then let  $\pi_s$ be the smallest value in $\beta^{(g-1)}$ and let $t$ be the rightmost horizontal gap in Figure~\ref{fig:2143_4132_decomp}.  Now define $\alpha$ to be everything weakly above $\pi_s$ and strictly left of $t$.  Note that this definition of $\alpha$ implies that $\alpha$ is nonempty and $\alpha_1 \neq 1$.  The same argument used in {\tt Case 1} shows that $\alpha\in \Av(4132)$.  Next, let $\beta$ be the permutation obtained from all the values in $\pi$ lying  strictly below $\pi_s$. Observe that
\begin{align}\label{eq:dec_524361}
\pi = (\alpha'\ominus_i \beta)\oplus 1\ldots m,
\end{align}
where $\alpha' = \alpha\oplus 1$, $i\leq\lead(\beta)$ and $\beta_{i+1}\beta_{i+2}\ldots\in\CC(524361) \setminus \Av(132)$. 

Lastly, a straightforward check shows that for any such $\alpha$ and $\beta$ as described the permutation constructed as (\ref{eq:dec_524361}) is in  $\CC(524361)\setminus \CC(4132)$ with at least two horizontal gaps.  
\qed

\bigskip

In order to translate this decomposition into a functional equation we again use the catalytic variable $t$ to mark the number of leading maxima.  We define 

\begin{align}
A_2(t,x)&=\sum_{\pi\in\CC(524361)}{x^{|\pi|}t^{\lead(\pi)}},\label{eq:A_2:def}\\
Y(t,x)&=\sum_{\pi\in\CC(4132)}{x^{|\pi|}t^{\lead(\pi)}},\label{eq:Y}\\
Z(t,x)&=\sum_{\substack{\pi\in\CC(4132)\\ \pi\ne\emptyset,\ \pi_1\neq 1}}{x^{|\pi|}t^{\lead(\pi)}}\label{eq:Z}.
\end{align}
Using these variables, we now obtain the functional equation
\begin{equation}\label{eq:A_2:functional}
A_2(t,x)=Y(t,x)+\frac{tx}{1-x}D(t,x)+\frac{x}{1-x}D(t,x) Z(t,x),
\end{equation}
where 
\begin{equation*}
D(t,x)=\frac{A_2(1,x)-tA_2(t,x)}{1-t}-\frac{1}{1-tx}C\left(\frac{x}{1-tx}\right).
\end{equation*}
To be clear, the first term in  $A_2(t,x)$ corresponds to {\tt Case 1} whereas the last two terms correspond to {\tt Case 2}. Specifically,  the second term counts all $\pi\in \CC(524361)\setminus \Av(132)$ with exactly one horizontal gap and the third term counts all those with at least two horizontal gaps.  In this last term, the fact that $\alpha$ in the decomposition is nonempty and such that $\alpha_1 \neq 1$ is mirrored by the presence of $Z(t,x)$ in this final term, instead of $Y(t,x)$.  Lastly, $D(t,x)$ represents all the ways to perform {\tt extraction} on the elements in $\CC(524361)\setminus \Av(132)$.    

We now show that the large Schr\"oder numbers are the unique solution to the functional equation (\ref{eq:A_2:functional}).  A key ingredient needed in order to work out the algebra is the next lemma, whose proof we postpone to Subsection~\ref{sec:4132}.

\begin{lemma}\label{lem:enum_Y_Z}
The generating function $Y(t,x)$ and $Z(t,x)$ are given by the following expression 
\begin{equation} \label{eq:Y(t,x)}
Y(t,x)=\frac{1-tx+(tx-x)C^*}{(1-xC^*)(1-tx)},
\end{equation}
and 
\begin{equation} \label{eq:Z(t,x)}
Z(t,x)=\frac{tx(C^*-1)}{1-xC^*},
\end{equation}
where $C^* = C\left(\frac{x}{1-tx}\right)$.
\end{lemma}

Letting $A_2=A_2(t,x)$ and $B=A_2(1,x)$, we obtain, after a bit of algebra, 
\begin{equation} \label{eq:A-ker}
(tx-1)(t^2x+(t-1)xC^*-(t-1))A_2=(t-1)xC^*-(tx-1)(txB-(t-1)).
\end{equation}

Applying the kernel method, we set
\[
(tx-1)(t^2x+(t-1)xC^*-(t-1))=0.
\]
Letting $tx-1=0$ in the right-hand side of \eqref{eq:A-ker} yields $xC^*=0$, which is impossible. Thus, $t=t(x)$ is a solution of
\begin{equation} \label{eq:ctx}
t^2x+(t-1)xC^*-(t-1)=0.
\end{equation}
Now observe that the right-hand side of \eqref{eq:A-ker} is
\[
(t-1)xC^*+t^2x-t+1+(tx-1)txB-tx=0,
\]
so by \eqref{eq:ctx} we get
\begin{equation} \label{eq:atx}
B=\frac{1}{1-tx}.
\end{equation}
Solving \eqref{eq:ctx} for $t$, we see that the only such solution is
\[
t=\frac{1+x-\sqrt{1-6x+x^2}}{4x},
\]
which is, incidentally, the ordinary generating function for the little Schr\"oder numbers. This implies that
\[
B=\frac{1}{1-tx}=\frac{3-x-\sqrt{1-6x+x^2}}{2},
\]
and thus $\CC(524361)$ is enumerated by the large Schr\"oder numbers.

\subsection{The class $\CC(546132)$}

The similarities between this class and $\CC(524361)$ are immediately apparent. To see this, we begin, as usual, by considering Figure~\ref{fig:2143_4132_decomp}.  Provided $g\geq 2$, we claim that $\beta^2,\ldots,\beta^g\in \Av(132)$.  If not, then for some $1<i$, $\beta^i$ contains an occurrence of 132.  Setting $\ell$ and $t$ to be the first and second horizontal gaps respectively,  we see that $\pi_\ell$, any value in $\beta^1$, $\pi_t$, and $\beta^i$ would create an occurrence of 546132.  Just as in the previous section, we now have the following two cases.

\smallskip
{\tt Case 1: $\beta^1\in \Av(132)$}

\smallskip

As in {\tt Case 1} in the previous section, $\pi\in \CC(4132)\subset\CC(546132)$. \qed

\medskip

{\tt Case 2: $\beta^1\notin \Av(132)$}

\smallskip

If $\pi$ has exactly one horizontal gap, then 
$$\pi = (1\ominus_i \beta)\oplus 1\ldots m,$$
where $i\leq \lead(\beta)$ and $\beta_{i+1}\beta_{i+2}\ldots\in\CC(546132)\setminus \Av(132)$.

If $\pi$ has at least two horizontal gaps, then let $\ell=\lead(\pi)$ and $t$ be the first and second horizontal gaps in $\pi$ respectively.  Let $\pi_r$ be  the largest value in the block labeled $\beta^2$  in Figure~\ref{fig:2143_4132_decomp}.  Now define $\alpha$ to be the permutation which is strictly to the left of $t$ and strictly above $\pi_r$. As  $\alpha$ has exactly one horizontal gap, then
$$\alpha = (1\ominus_i \gamma)\oplus 1\ldots m,$$
for some $\gamma$ and $i$ such that $\beta^1 = \gamma_{i+1}\gamma_{i+2}\ldots \in \CC(546132)\setminus \Av(132)$.

Further, let $\beta$ be the result of deleting the values corresponding to $\alpha$ in $\pi$.  (That is $\beta$ is the permutation defined by the values of $\pi$ which are either weakly below $\pi_r$ or weakly to the right of $t$.) As $\beta^2,\ldots,\beta^g\in\Av(132)$, Lemma~\ref{lem:char:4132} now implies that $\beta\in\CC(4132)$ with the property that $\beta$ is nonempty and $\beta_{\lead(\beta)}-1 \neq\beta_{\lead(\beta)-1}$, a point we will return to shortly.   We now see that $\pi$ is obtained by inflating the value $\beta_{\lead(\beta)}$ with  $\alpha\oplus 1$. 

Lastly, a straightforward check shows that for any such $\alpha$ and $\beta$ as just described, the permutation constructed by inflating the value $\beta_{\lead(\beta)}$ with  $\alpha\oplus 1$ is in $\CC(524361)\setminus \CC(4132)$ with at least two horizontal gaps.\qed

\bigskip

In the decomposition of $\pi$ described in {\tt Case 2}, we see that $\beta$ is a member of the following subset of $\CC(4132)$,
$$\mathcal{A}=\{\sigma\in\CC(4132)\ |\ \sigma\ne\emptyset,\ \sigma_{\lead(\sigma)}-1 \neq\sigma_{\lead(\sigma)-1} \}.$$
In order to translate the above decompositions into a functional equation, we will certainly need to enumerate this set. In fact, we have unwittingly already done so.  We know from Equation~\ref{eq:Z} that $Z(t,x)$ is the generating function for the set
$$\mathcal{B} = \{\sigma\in\CC(4132)\ |\ \sigma\ne\emptyset,\ \sigma_1 \neq 1 \},$$
and we claim that it also the generating function for $\mathcal{A}$ as well.  To see this, we construct a (length preserving) bijection between the sets $\mathcal{A}$ and $\mathcal{B}$.  Let  $\sigma\in \mathcal{A}$ and identify the largest leading maximum $i$ in $\sigma$ such that $\sigma_j = j$ for all $j\leq i$.    Now move these first $i$ leading maxima and place them immediately below $\sigma_{\ell(\sigma)}$.  This is certainly a bijection from $\mathcal{A}$ to $\mathcal{B}$ that preserves the number of leading maxima.  Thus $Z(t,x)$ counts $\mathcal{A}$ as well.  

We are now in a position to combine all these pieces into a functional equation.  Let   
\begin{align}\label{eq:A_3}
A_3(t,x)&=\sum_{\pi\in\CC(546132)}{x^{|\pi|}t^{\lead(\pi)}}.
\end{align}
It follows immediately from the above decomposition that 
$$A_3(t,x) = Y(t,x) + \frac{txD(t,x)}{1-x} + \frac{xD(t,x)}{1-x}Z(t,x),$$
where $Y(t,x)$ and $Z(t,x)$ are defined by (\ref{eq:Y}) and  (\ref{eq:Z}), respectively and 
$$D(t,x)=\frac{A_3(1,x)-tA_3(t,x)}{1-t}-\frac{1}{1-tx}C\left(\frac{x}{1-tx}\right).$$  
In particular, the first term encodes {\tt Case 1}, whereas the second term encodes {\tt Case 2} for exactly one horizontal gap, and the third term encodes this case for at least two horizontal gaps.  Finally, we would draw the reader's attention to the absence of the factor of $t$ in the numerator of the third term.  This corresponds to the fact that when we inflate $\beta$  with $\alpha$ the leading maximum corresponding to $\ell(\beta)$ in the resulting permutation is no longer a leading maximum.

  As this is the same functional equation as in (\ref{eq:A_2:functional}), we immediately conclude that $\CC(546132)$ is counted by the large Schr\"oder numbers.

\subsection{The class $\CC(4132)$}\label{sec:4132}
In order to complete our enumeration of the classes $\CC(524361)$ and $\CC(546132)$, it only remains to prove Lemma~\ref{lem:enum_Y_Z}.  We do precisely that in this section.  To aid the reader, recall that 
$$
Y(t,x)=\sum_{\pi\in\CC(4132)}{x^{|\pi|}t^{\lead(\pi)}}\qquad\textrm{and}\qquad
Z(t,x)=\sum_{\substack{\pi\in\CC(4132)\\ \pi\ne\emptyset,\ \pi_1\neq 1}}{x^{|\pi|}t^{\lead(\pi)}}.$$
Moreover, Lemma~\ref{lem:enum_Y_Z} stated that 
$$
Y(t,x)=\frac{1-tx+(tx-x)C^*}{(1-xC^*)(1-tx)}\qquad\textrm{and}\qquad
Z(t,x)=\frac{tx(C^*-1)}{1-xC^*},$$
where $C^* = C\left(\frac{x}{1-tx}\right)$.

\begin{proof}[Proof of Lemma~\ref{lem:enum_Y_Z}]

Our enumeration of the class $\CC(4132)$ will employ techniques similar to those in the previous two subsections. Consequently, we will be brief.  

For any $\pi\in \CC(4132)$, we (again) consider the number of horizontal gaps in $\pi$. If $\pi$ has no horizontal gaps, then $\pi$ is the increasing permutation.  

If $\pi$ has exactly one horizontal gap, then it is of the form 
$$(1\ominus_i \beta)\oplus 1\ldots m,$$
where $i<|\beta|$ and $\beta_{i+1}\beta_{i+2}\ldots \in \Av(132)$.  Such permutations are counted by
$$\frac{tx \left(C\left(\frac{x}{1-tx}\right)-1\right)}{(1-tx)(1-x)},$$
where $C(x)$ is the generating function for the Catalan numbers (and the set $\Av(132)$).  Note that the construction   
$\displaystyle C\left(\frac{x}{1-tx}\right)$ effectively counts the number of additional LR-maxima inserted \emph{below} each row of a 132-avoiding permutation.  

Lastly, if $\pi$ has at least two horizontal gaps, then it decomposes as
$$(\alpha\ominus_i \beta)\oplus 1\ldots m,$$
where $\alpha$ is not increasing (so as to have at least one horizontal gap),  $\alpha\in \CC(4132)$,  and $i<|\beta|$, and $\beta_{i+1}\beta_{i+2}\ldots \in \Av(132)$.  It readily follows that such permutations are counted by
$$\frac{x}{1-x}\left(Y(t,x)-\frac{1}{1-tx}\right) \left(C\left(\frac{x}{1-tx}\right)-1\right).$$

Combining these cases yields the following functional equation:
$$
Y(t,x)=\frac{1}{1-tx}+
\frac{tx (C^*-1)}{(1-tx)(1-x)}+
\frac{x}{1-x}\left(Y(t,x)-\frac{1}{1-tx}\right) (C^*-1),
$$
whose solution is the desired equation.  Finally, as 
$$Z(t,x) = (1-tx) Y(t,x)-1,$$
we easily obtain our expression for $Z(t,x)$.  
\end{proof}

Before concluding this subsection, we note that letting $t=1$ we obtain
\[
Y(1,x)=\frac{1}{1-xC^*}=\frac{2}{1+x+\sqrt{(1-x)(1-5x)}},
\]
in other words, $|\CC_n(4132)|=\href{http://oeis.org/A033321}{\mathrm{A033321}}(n)$, see \cite{Sloane}.

\section{The class $\CC(263514)$} \label{sec:263514}

We now turn our attention to our last remaining class.  The techniques needed to enumerate $\CC(263514) = \Av(2143,3142, 263514)$ are not the same as those used in the previous sections but instead mirror those used by Burstein and Pantone in~\cite{BP}.  As mentioned in the introduction, we make use of simple permutations and the inflation construction.  For readers unfamiliar with these notions, we pause to define these terms.  

First define $[i,j] = \{i,i+1,\ldots, j\}$, for $i\leq j$.  We say $[i,j]$ is an \emph{interval} of length $j-i+1$ in a permutation $\sigma$ provided the values $\{\sigma_i,\ldots, \sigma_j\} = [a,b]$, for some $a\leq b$.  A permutation $\sigma$ is \emph{simple} provided that its only intervals have length $1$ or $n$.  Lastly, if $A$ is any set of permutations, then we denote by $\Si(A)$ the subset of all simple permutations in $A$.

Observe that the permutation $315462$ is not simple because it contains the interval $[3,5]$. That said, if we ``deflate" the corresponding string of values $546$ into a $4$ we obtain the simple permutation $3142$.   This idea motivates our next definition.   If $\sigma$ is any permutation of length $k$ and $\{\rho^{(i)}\}_{i=1}^k$ is a sequence of nonempty permutations, then the \emph{inflation} of $\sigma$ by $\{\rho^{(i)}\}_{i=1}^k$, written as $\sigma[\rho^{(1)}, \rho^{(2)}, \ldots, \rho^{(k)}]$, is the permutation of length $|\rho^{(1)}| + \cdots + |\rho^{(k)}|$ such that each entry $\sigma_i$ is replaced by the permutation $\rho^{(i)}$. For example,

\begin{center}
\begin{tikzpicture}[scale=.5,baseline=(current bounding box.center)]
\fill[blue!10!white] (0,4) rectangle (3,7);
\fill[blue!10!white] (3,3) rectangle (4,4);
\fill[blue!10!white] (4,7) rectangle (6,9);
\fill[blue!10!white] (6,0) rectangle (9,3);

\draw[step=1cm,lightgray,very thin] (0,0) grid (9,9);

\node at (-5,5) {$3241[123,1,12,123]\ =$};
\node at (0.5,4.5) {$\times$};
\node at (1.5,5.5) {$\times$};
\node at (2.5,6.5) {$\times$};
\node at (3.5,3.5) {$\times$};
\node at (4.5,7.5) {$\times$};
\node at (5.5,8.5) {$\times$};
\node at (6.5,0.5) {$\times$};
\node at (7.5,1.5) {$\times$};
\node at (8.5,2.5) {$\times$};

\end{tikzpicture}.
\end{center}

To guide the reader, we divide the remainder of this section into two parts.  In the first part, we characterize the simple permutations $\CC(263514)$ in terms of $132$-avoiding permutations.  We then demonstrate how the simples may be inflated to obtain arbitrary elements of $\CC(263514)$.  In the second part, we translate these descriptions into functional equations whose unique solution, we show, is the large Schr\"oder numbers.  

\subsection{The Simples}

In this section we show the unexpected result that the simples in $\CC(263514)$ are essentially  132-avoiding permutations with LR-maxima  strategically inserted to ``break up" intervals.    As this description is quite reminiscent of Lemma~\ref{lem:char:4132}, the statement of our next lemma is quite natural.  

\begin{lemma} \label{lem:263514-simples}
 $\Si(\CC(263514))=\Si(\CC(4132))$.
\end{lemma}

\begin{proof}
Clearly, $\Si(\CC(4132))\subseteq\Si(\CC(263514))$, since $\CC(4132)\subseteq\CC(263514)$.  For the other inclusion, fix $\sigma\in \Si(\CC(263514))$ and assume for a contradiction that it contains an occurrence of $4132$.  Let $\sigma_x\sigma_y\sigma_z\sigma_w$ be such an occurrence and consider Figure~\ref{fig:CharSimples}.  Note that the region 1 and 3 must be empty as $\sigma$ avoids $2143$ and $3142$.  Moreover, we can also assume that regions 2, 4, and 6 are also empty.  For example, if region 4 is not empty, then redefine $\sigma_x$ to be the smallest value in region 4.  Observe that if region 5 is empty then the $[y,w]$ is a nontrivial interval in $\sigma$, which is impossible.  Therefore, region 5 must contain some smallest value $\sigma_i$.  In this case we obtain the same contradiction, i.e., $[y,w]$ must be a nontrivial interval in $\sigma$.    This follows since regions 7 and 8 above $\sigma_i$ must be empty as $\sigma$ avoids 263514 and 3142.  We therefore conclude that $\sigma\in \CC(1432)$ as desired.  
\end{proof}

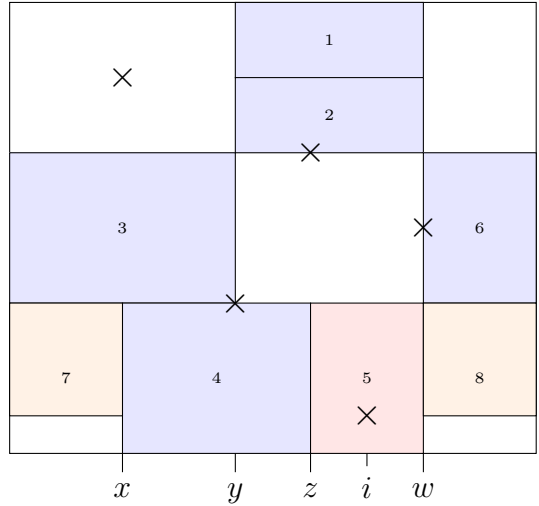
\begin{figure}
\begin{center}
\begin{tikzpicture}[scale=1]
\draw(0,0) rectangle (7,6);
\draw[fill=blue!10!white] (3,5) rectangle (5.5,6) node[midway] {\tiny 1};
\draw[fill=blue!10!white] (3,4) rectangle (5.5,5) node[midway] {\tiny 2};
\draw[fill=blue!10!white] (0,2) rectangle (3,4) node[midway] {\tiny 3};
\draw[fill=blue!10!white] (5.5,2) rectangle (7,4) node[midway] {\tiny 6};
\draw[fill=blue!10!white] (1.5,0) rectangle (4,2) node[midway] {\tiny 4};
\draw[fill=red!10!white] (4,0) rectangle (5.5,2) node[midway] {\tiny 5};
\draw[fill=orange!10!white] (0,0.5) rectangle (1.5,2);
\draw[fill=orange!10!white] (5.5,0.5) rectangle (7,2);
\node at (.75, 1) {\tiny 7};
\node at (6.25, 1) {\tiny 8};

\node at (1.5,5) {\large $\times$};
\node at (3,2) {\large $\times$};
\node at (4,4) {\large $\times$};
\node at (5.5,3) {\large $\times$};
\node at (4.75,.5) {\large $\times$};

\draw(1.5,0) -- (1.5,-.25) node[below] {$x$};
\draw(3,0) -- (3,-.25) node[below] {$y$};
\draw(4,0) -- (4,-.25) node[below] {$z$};
\draw(5.5,0) -- (5.5,-.25) node[below] {$w$};
\draw(4.75,0) -- (4.75,-.17) node[below] {$i$};
\end{tikzpicture}
\caption{A schematic showing the different regions of $\sigma$.}
\label{fig:CharSimples}
\end{center}
\end{figure}

To complete our characterization we define the following permutation statistic.  We say an index $i$ in a permutation $\pi$ is a \emph{bond} if $\pi_i+1 = \pi_{i+1}$ or $\pi_i-1 = \pi_{i+1}$ and denote the number of bonds in $\pi$ as $\bonds(\pi)$.    

For permutations of length greater than 2, not containing bonds is a necessary condition for being simple.  The next lemma states that this is also sufficient for the simples in $\CC(4132)$.  

Let $\T$ be the set of all permutations  obtained by the following construction.   For $n\geq 3$, choose any $\alpha\in \Av_n(132)$ such that $\alpha_1 = n$ and $\alpha_n = n-1$.  Now insert at most one leading maximum below rows $2,\ldots, n-1$. (We do not allow for an insertion to occur below the 1st or $n$th row as this creates either a sum decomposable permutation or a bond, respectively.) Moreover, insert exactly one leading maximum between any two rows that contain a bond.  For reference we provide an example in Figure~\ref{fig:4132Simples}.

\begin{figure}
\begin{center}
\begin{tikzpicture}[scale=.5, baseline=(current bounding box.center)]
\fill[blue!10!white] (0,3) rectangle (12,4);
\fill[red!10!white] (0,1) rectangle (12,2);
\fill[red!10!white] (0,6) rectangle (12,7);
\fill[blue!10!white] (0,8) rectangle (12,9);
\draw[step=1cm,lightgray,very thin] (0,0) grid (12,12);

\node at (1.5,3.5) {$\times$};
\node at (0.5,1.5) {$\times$};
\node at (2.5,6.5) {$\times$};
\node at (3.5,8.5) {$\times$};
\node at (4.5,11.5) {$\times$};
\node at (5.5,5.5) {$\times$};
\node at (6.5,7.5) {$\times$};
\node at (7.5,4.5) {$\times$};
\node at (8.5,9.5) {$\times$};
\node at (9.5,0.5) {$\times$};
\node at (10.5,2.5) {$\times$};
\node at (11.5,10.5) {$\times$};

\end{tikzpicture}.
\caption{This depicts the 132-avoiding permutation 84536127 after 4 leading maxima have been inserted.  The rows 2 and 7 highlighted in red show where an insertion of leading maxima was forced due to a bond, whereas rows 4 and 9 highlighted in blue represent the insertion of a maximum.}
\label{fig:4132Simples}
\end{center}
\end{figure}
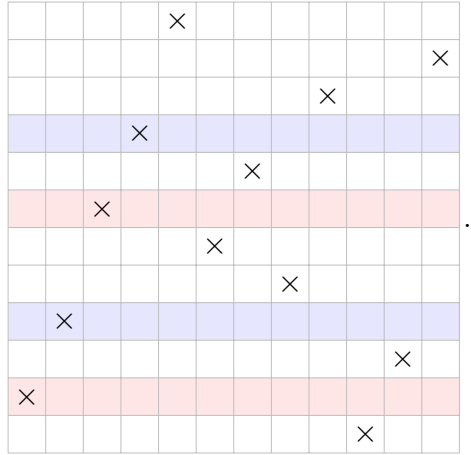

\begin{lemma}\label{lem:description_of_simples}
$\Si(\CC(4132)) = \T \cup \{1, 12,21\}$
\end{lemma}
\begin{proof}

A direct application of Lemma~\ref{lem:char:4132} shows that $\T\subset\CC(4132)$.  Therefore in order to show that  $\T \cup \{1, 12,21\}\subset \Si(\CC(4132))$ we must only show that $\T$ consists of only simple permutations.  To this end, fix $\sigma\in \T$ and where $n\geq 4$ and $\ell =\ell(\sigma)$.    First observe that any 132-avoiding permutation of length greater than 1 must contain at least one bond. To see that $\sigma$ is simple, consider any set $[i,j]\neq [1,n]$, with $i<j$.  We now have three cases.

\smallskip

{\tt Case 1:} $\ell\in [i,j]$

\smallskip

From our construction of $\sigma$, it immediately follows that $\sigma_\ell = n$ and $\sigma_n = n-1$.  This implies that if $j<n$, then $[i,j]$ is not an interval in $\sigma$.  Now consider the case that $j=n$.  In order for $[i,j]$ to be an interval in this case, it must contain all the leading maxima in $\sigma$. (Here we use the fact that $\sigma_1 \neq 1$.)  But this is not possible as $1\notin [i,j]$.

\medskip

{\tt Case 2:} $j<\ell$

\smallskip

In this case, the set $[i,j]$ consists entirely of leading maxima.  Moreover, it follows from our construction of $\sigma$ that any two leading maxima in $\sigma$  straddle some value of $\sigma_i$  where $\ell<i$.  Therefore, $[i,j]$ cannot be an interval in $\sigma$.  
    
\medskip

{\tt Case 3:} $\ell<i$

\smallskip

By construction the permutation $\sigma_i\ldots\sigma_j$ is a 132-avoiding permutation.  By our initial observation, this permutation contains a bond.  Our construction of $\sigma$ then prevents $[i,j]$ from being an interval in $\sigma$.

It now remains to show that $\Si(\CC(4132))\subset \T\cup \{1, 12, 21\}$.  As $S_3$ contains no simple permutations, we may fix  $\sigma\in \Si(\CC_{\geq 4}(4132))$.  Set $\ell = \lead(\sigma)$.   From Figure~\ref{fig:2143_4132_decomp}, and Lemma~\ref{lem:char:4132} it will immediately follow that $\sigma\in \T$ provided that $\sigma_n = n-1$.   To see that this must hold assume not.  Then $\sigma$ would have been constructed from a 132-avoiding permutation of the form $\alpha\ominus \beta,$ where $\beta\neq \emptyset$. Then in $\sigma$ the values corresponding to $\alpha$ and the values inserted between rows of $\alpha$ constitute a nontrivial interval in $\sigma$.  We conclude the $\sigma\in\T$ as needed.

\end{proof}

Having completed our characterization of the simples in $\CC(263514)$, we now consider how to (uniquely) obtain an arbitrary element in this class from one of its simples.  This task will require a new definition which we state next.  A \emph{LR-minimum} is an index $i$ in a permutation such that $\pi_j>\pi_i$ for all $j<i$.  

\begin{lemma}\label{lem:inflation-1}
Fix $\sigma\in \Si(\CC_n(263514))$ and  $\rho^{(1)},\ldots, \rho^{(n)} \in \CC(263514)$.  Then 
$$\sigma[\rho^{(1)},\ldots, \rho^{(n)}]\in \CC(263514),$$
provided that $\rho^{(i)}$ is an increasing pattern whenever $\sigma_i$ is the $1$ in a $132$-pattern or the $3$ in a $213$-pattern.  Moreover,  this is the only way to inflate $\sigma$ into an element of $\CC(263514)$.   
\end{lemma}

\begin{proof}
As $\sigma \in \CC(263514)$ and both the patterns 3142 and 263514 are simple, it follows that $\sigma[\rho^{(1)},\ldots, \rho^{(n)}]$ will always avoid 3142 and 265314 for any $\rho^{(i)}\in \CC(263514)$.   Therefore we only need to consider the pattern 2143.  It is clear that we can always inflate a given position in $\sigma$ with the monotone pattern $\id_k$ since this cannot create an occurrence of 2143.  Moreover, a position $i$ in $\sigma$ can  be inflated by an arbitrary element in $\CC(263514)$ if and only if $\pi_i$ is neither the $1$ in a 132-pattern nor the 3 in a $213$-pattern.  
\end{proof}

The next lemma serves to identify which indices $i$ in $\sigma$ have the property that either $\sigma_i$ is  a $1$ in a 132-pattern or a 3 in a $213$-pattern.

\begin{lemma}\label{lem:inflation-2}
Fix $\sigma\in \Si(\CC(263514))$ and set $\ell = \ell(\sigma)$.   The value $\sigma_i$ is a $1$ in a $132$-pattern if and only if $i<\ell$.  Additionally, the value $\sigma_i$ is a $3$ in a $213$-pattern if and only if  $\ell<i$ and $i-\ell$ is not a LR-minimum in the $132$-avoiding permutation $\sigma_{\ell+1}\sigma_{\ell+2}\ldots$.  
\end{lemma}

\begin{proof}
Set $n= |\sigma|$.  To see the first claim, recall from our characterization of $\sigma$ that $\sigma_\ell = n$ and $\sigma_n = n-1$.  It now follows $\sigma_i\sigma_\ell\sigma_n$ is a 132.  Indeed this is the only possibility since $\sigma_\ell\ldots\sigma_n$ is a 132-avoiding permutation by our characterization.  

For the second claim, let us begin by assuming $i-\ell$ is a LR-minimum in $\sigma_{\ell+1}\sigma_{\ell+2}\ldots\sigma_{n}$.  From our characterization of $\sigma$, this implies that in $\sigma$ all the values southwest of $\sigma_i$ are increasing.  Therefore $\sigma_i$ cannot be such a 3.  Likewise, if $i<\ell$, then for the same reason $\sigma_i$ is not a 3 in an occurrence of a 213-pattern.    For the other direction, assume $\ell<i$ and $i-\ell$ is not a LR-minimum in the 132-avoiding permutation $\sigma_{\ell+1}\sigma_{\ell+2}\ldots$.  This means that in $\sigma$ there exists some $\ell<y$ such that $\sigma_y<\sigma_i$.  As $\sigma$ is simple and hence contains no bonds, then it must have some index $x$ such that  $y<x<i$ or such that $\sigma_y<\sigma_x<\sigma_i$.  Since $\sigma_{\ell+1}\sigma_{\ell+2}\ldots$ avoids 132, a straightforward check shows that the only possibility is for either $\sigma_x\sigma_y\sigma_i$ or $\sigma_y\sigma_x\sigma_i$ to be a 213-pattern.  This completes our proof.  
\end{proof}

Before closing this section, the following standard proposition, due to Albert and Atkinson, shows that every permutation arises from the inflation of a (unique) simple permutation.  This result along with the previous results in this section together imply that every element in $\CC(263514)$ arises (uniquely) from the inflation of a permutation in $\Si(\CC(263514))$ according to the prescriptions given in Lemmas~\ref{lem:inflation-1} and~\ref{lem:inflation-2}.

\begin{prop}[\cite{AA}]
Given a permutation $\pi$, there exists a unique simple permutation $\sigma$ of length $k$ such that $\pi = \sigma[\rho^{(1)},\ldots,\rho^{(k)}]$.  When $\sigma \neq 12, 21$, the $\rho^{(i)}$ are uniquely determined.  When $\sigma = 12$ (respectively, $21$), the $\rho^{(i)}$ are uniquely determined provided we insist that $\rho^{(1)}$ is sum indecomposable (respectively, skew-sum indecomposable).
\end{prop}

\subsection{Enumeration}

We now transform the descriptions in the previous subsection into functional equations.  The generating function for the simple permutations $\Si(\CC(263514))$ is
$$s(u,x) = \sum_{\substack{\sigma\in\ \Si(\CC_n(263514))\\n\geq 4}} u^{n-\lrmin(\widehat{\sigma})-1}x^{\lrmin(\widehat{\sigma})+1} ,$$ where $\widehat{\sigma}$ is the result of stripping $\sigma$ of its leading maxima and $u$ marks the number of positions that must be inflated by the increasing pattern whereas $x$ marks the number of positions that may be inflated by an arbitrary permutation in the class.  Lemma~\ref{lem:263514-simples}  and \ref{lem:description_of_simples} now suggest setting
$$C(t,u,x) = \sum_{n\geq 2}\sum_{\substack{\pi\in \Av_n(132) \\ \pi_n = n}}  t^{\bonds(\pi)}u^{\lrmin(\pi)}x^n,$$ with $t$ marking the number of bonds, and $u$ marking the number of LR-minima.  Additionally, these lemmas imply the  equation:
$$s(u,x) = \frac{x}{1+ u}C\left(\frac{u}{1+u},\frac{x}{u},u(1+u)\right).$$

To find an expression for $C(t,u,x)$ we set 
$$h(t,u,x) = \sum_{n\geq 0}\sum_{\substack{\pi\in \Av_n(132) \\ \pi_n \neq n}} x^n t^{\bonds(\pi)}u^{\lrmin(\pi)},$$
and
$$g(t,u,x) = \sum_{n\geq 0}\sum_{\substack{\pi\in \Av_n(132) \\ \pi_1 \neq n}} x^n t^{\bonds(\pi)}u^{\lrmin(\pi)},$$
since then
\[
C(t,u,x) = \frac{x(h-1)}{1-xt} + \frac{utx^2}{1-xt}.
\]

Using the classic decomposition of 132-avoiding permutations, we see that $h$ and $g$ satisfy the following functional equations
\begin{align*}
h(t,u,x)=1&+ x \left(\frac{(h-1) t x}{1-t x}+h+\frac{t u x}{1-t x}-1\right) \left(\frac{ux}{1-t u x}g+g-1\right)\\
&+u x \left(\frac{tux}{1-t u x}g+g-1\right)
\end{align*}
and 
$$g(t,u,x) =1+ x \left(\frac{(h-1) t x}{1-t x}+h+\frac{t u x}{1-t x}-1\right) \left(\frac{ u x}{1-t u x}g+g\right).$$

Outsourcing the computations to a CAS, we see that 
$$s(u,x) = -\frac{x \left(-1+u+3 u x+u x^2 -\sqrt{1+u^2 \left(x^2+x+1\right)^2-2 u (x^2+3x+1)}\right)}{2 (u+1) (x+1)}.$$

We are now ready to count all permutations in $\CC(263514)$. Set
$$f(x) = \sum_{n\geq 1}  |\CC_n(263514)| x^n$$
and let $f_\oplus$ and $f_\ominus$ be the generating function for all sum and skew-sum decomposable permutations in $\CC(263514)$, respectively.  As none of the patterns 2143, 3142, and 263514 are skew-sum decomposable it follows that $f_\ominus = (f - f_\ominus)f$.  Therefore
 $$f_\ominus = \frac{f^2}{1+f}.$$  
On the other hand, since 2143 is the only sum-decomposable pattern involved, it follows that any sum decomposable permutation is either of the form $1\oplus \pi$ or $\pi\oplus 1$.  It follows that $$f_\oplus = 2xf - x^2(f+1)$$ 
where we take care not to overcount the permutations that are of both forms.  Putting this together, we obtain the long sought functional equation:
$$f = x + f_\ominus + f_\oplus + s\left(\frac{x}{1-x},f\right).$$

Using a CAS we see that the large Schr\"oder numbers are a solution to this functional equation, proving the last case.

\end{document}